\documentclass[12pt, leqno]{article}
\usepackage{amsmath}
\usepackage{amsthm}
\usepackage{amssymb}
\usepackage{latexsym}
\usepackage{graphicx}
\allowdisplaybreaks[1]
\usepackage{amsfonts}
\usepackage{ascmac}
\usepackage{color}
\usepackage{enumerate}
\theoremstyle{definition}
\theoremstyle{plain}
\allowdisplaybreaks[1]
\date{}
\newtheorem{Thm}{Theorem}[section]
\newtheorem{Prop}[Thm]{Proposition}
\newtheorem{Lemma}[Thm]{Lemma}

\newcommand{\brho}{\bar{\rho}}
\newcommand{\bbrho}{\bar{\bar{\rho}}}

\newcommand{\dt}{\Delta t}
\newcommand{\dx}{\Delta x}

\newcommand{\p}{\partial}

\newcommand{\dis}{\displaystyle}

\newcommand{\norm}{\parallel}

\newcommand{\Z}{{\mathbb Z}}

\newcommand{\T}{{\mathbb T}}

\newcommand{\R}{{\mathbb R}}

\def\text#1{\mbox{#1 }}

\title{\bf Stochastic and Variational Approach \\to the Lax-Friedrichs Scheme}
\author{Kohei Soga
\footnote{Department of Pure and Applied Mathematics, Waseda University, Tokyo 169-8555, Japan \qquad\qquad
(kohei-math@toki.waseda.jp). Supported by Grant-in-Aid for JSPS Fellows (20-6856).}}
\begin{document}
\maketitle
\begin{abstract} 
\noindent We present a stochastic and variational aspect of the Lax-Friedrichs scheme applied to hyperbolic scalar conservation laws. This is a finite difference version of  Fleming's results ('69) that the vanishing viscosity method is characterized by stochastic processes and calculus of variations. We convert the difference equation into that of the Hamilton-Jacobi type and introduce corresponding calculus of variations with random walks. The stability of the scheme is obtained through the calculus of variations. The convergence of approximation is derived from the law of large numbers in hyperbolic scaling limit of random walks. The main advantages due to our approach are the following: Our framework is basically pointwise convergence, not $L^1$ as usual, which yields uniform convergence except ``small'' neighborhoods of shocks; The convergence proof is verified for arbitrarily large time interval, which is hard to obtain in the case of flux functions of general types depending on both space and time; The approximation of characteristics curves is available as well as that of PDE-solutions, which is particularly important for applications of the Lax-Friedrichs scheme to the weak KAM theory.  
\noindent  
 
\medskip

  \medskip

\end{abstract}
\setcounter{section}{0}
\setcounter{equation}{0}
\section{Introduction}
The Lax-Friedrichs scheme is one of the oldest, simplest and most universal technique of computing PDEs. There is the huge literature on the stability and convergence of the Lax-Friedrichs scheme based on the $L^1$-framework, particularly for shock waves.   
In this paper we investigate the Lax-Friedrichs scheme in terms of scaling limit of random walks and calculus of variations and present several useful new results, relating the scheme to principles of probability theories and calculus of variations. We refer to only the Lax-Friedrichs scheme applied to hyperbolic scalar conservation laws. However our results imply that finite difference methods applied to evolution equations are likely to possess stochastic or variational structures. It is a meaningful effort to investigate the Lax-Friedrichs scheme or other schemes applied to PDEs of various types by similar approach to ours.    

We consider initial value problems of the inviscid hyperbolic  scalar conservation law
\begin{eqnarray}\label{CL}
\left\{
\begin{array}{lll}
&\dis u_t+H(x,t,c+u)_x=0\,\,\,\,\mbox{in $\T\times(0,T]$,}\medskip\\
&u(x,0)=u^0(x)\in L^\infty(\T)\,\,\,\,\mbox{on $\T$},\quad\int_\T u^0(x)dx=0,
\end{array}
\right.
\end{eqnarray}
where $c$ is a parameter varying within an interval $[c_0,c_1]$ and $\T:=\R/\Z$ is the standard torus. The assumptions for the flux function $H$ are the following (A1)-(A4):

(A1) $H(x,t,p):\T^2\times\R\to\R$, $C^2$ \quad  (A2) $H_{pp}>0$  \quad
 (A3) $\dis \lim_{|p|\to+\infty}\frac{H(x,t,p)}{|p|}=+\infty$.
 
 \noindent By (A1)-(A3), we have the Legendre transform $L(x,t,\xi)$ of $H(x,t,\cdot)$, which is now given by 
 $$L(x,t,\xi)=\sup_{p\in\R}\{\xi p -H(x,t,p)\}$$
 and satisfies 

(A1)' $L(x,t,\xi):\T^2\times\R\to\R$, $C^2$ \quad  (A2)' $L_{\xi\xi}>0$  \quad 
(A3)' $\dis \lim_{|\xi|\to+\infty}\frac{L(x,t,\xi)}{|\xi|}=+\infty$.
 
\noindent The last assumption is 
 
 (A4) There exists $\alpha>0$ such that $|L_x|\le\alpha(|L|+1)$. 
 
\noindent  Throughout this paper, $\T$-dependency is identified with $\R$-dependency with $\Z$-periodicity and $\T$ with $[0,1)$. (A1) and (A2) are standard in the theories of conservation laws. (A3) is necessary, when we introduce a variational approach stated below to our problems. (A4) is used for derivation of boundedness of minimizers of the variational problems. The problems (\ref{CL}) appear not only in continuum mechanics but also in Hamiltonian and Lagrangian dynamics generated by $H$ and $L$ 
\cite{Fathi97-1}, \cite{JKM}, \cite{WE}, \cite{Sob}. In the latter case the periodic setting is standard. We remark that the whole space setting is also available with additional assumptions for $H$ required by variational techniques.    

It is sometimes very convenient to introduce initial value problems of Hamilton-Jacobi equations which are equivalent to (\ref{CL})
\begin{eqnarray}\label{HJ}
\left\{
\begin{array}{lll}
&\dis v_t+H(x,t,c+v_x)=h(c)\,\,\,\,\mbox{in $\T\times(0,T]$,}\medskip\\
&v(x,0)=v^0(x)\in Lip(\T)\,\,\,\,\mbox{on $\T$}, \quad
\end{array}
\right.
\end{eqnarray}
where $h:[c_0,c_1]\to\R$ is a given function.  We consider (\ref{CL}) and (\ref{HJ}) in the class of generalized solutions called entropy solutions and viscosity solutions respectively. Such solutions exist in $C^0((0,T];L^\infty(\T))$ and $Lip(\T\times(0,T])$. If $u^0=v^0_x$, then the entropy solution $u$ of (\ref{CL}) and the viscosity solution $v$  of (\ref{HJ}) satisfy $u=v_x$ ($u$ is considered as a representative element). From now on we always assume that $u^0=v^0_x$.  

One of the central achievements in the analysis of (\ref{CL}) and (\ref{HJ}) is that they are closely related to the deterministic calculus of variations: The value of $v$ at each point $(x,t)\in\T\times(0,T]$ is given by 
 \begin{eqnarray}\label{value-func}
v(x,t)=\inf_{\gamma\in AC,\,\,\gamma(t)=x}\left\{ \int^t_0 L^c(\gamma(s),s,\gamma'(s))ds+v^0(\gamma(0)) \right\}+h(c)t,
\end{eqnarray}
where $AC$ is the family of absolutely continuous curves $\gamma:[0,t]\to\R$ and $L^c(x,t,\xi):=L(x,t,\xi)-c\xi$ is the Legendre transform of $H(x,t,c+\cdot)$ (see e.g. \cite{Cannarsa}). We can find a minimizing curve $\gamma^\ast$ of (\ref{value-func}), which is a backward characteristic curve of $u,v$ and a $C^2$-solution of the Euler-Lagrange equation associated with the Lagrangian $L^c$. If the point $(x,t)$ is a regular point of $v$ (i.e. there exists $v_x(x,t)$), then the value $u(x,t)=v_x(x,t)$ is given by
\begin{eqnarray}\label{value-func-CL}
u(x,t)=\int^t_0L^c_x(\gamma^\ast(s),s,\gamma^\ast{}'(s))ds+u^0(\gamma^\ast(0)),
\end{eqnarray}
 where $u^0$ is supposed to be rarefaction-free, or equivalently $v^0$ is semiconcave. (Otherwise $u^0(\gamma^\ast(0))$ needs to be replaced with $L^c_\xi(\gamma^\ast(0),0,\gamma^\ast{}'(0))$. See Lemma \ref{lemma-regularity}.) We remark that, since $v$ is Lipschitz, almost every point are regular and (\ref{value-func-CL}) is valid for almost every point.  
 
 The representation formula (\ref{value-func}) is a strong tool  not only in the analysis of (\ref{CL}) and (\ref{HJ}) but also in many applications of them to other fields such as optimal controls \cite{Fleming-Soner} and dynamical systems \cite{Fathi-book}. 
 It should be noted that {\it the variational approach to (\ref{CL}) and (\ref{HJ}) based on (\ref{value-func}) and  (\ref{value-func-CL}) also contributes approximation theories of (\ref{CL}) and (\ref{HJ}) by the vanishing viscosity method and the finite difference method.} The first case is announced by Fleming \cite{Fleming} and the latter case is the theme of this paper. 

First we recall the results of Fleming. Let us consider initial value problems of 
\begin{eqnarray}\label{CL-nu}&&
u^\nu_t+H(x,t,c+u^\nu)_x=\nu u^\nu_{xx}, \quad 
\\ \label{HJ-nu}
&& v^\nu_t+H(x,t,c+v^\nu_x )=h(c )+\nu v^\nu_{xx}\quad\,\,\, (\nu>0)
\end{eqnarray}
with the same setting as (\ref{CL}) and (\ref{HJ}).  The solutions $u^\nu$ and $v^\nu$ are also related to  calculus of variations which are not deterministic but   stochastic: The value of $v^\nu$ at each point $(x,t)\in\T\times(0,T]$ is given by 
\begin{eqnarray}\label{stochastic-value-func}
v^\nu(x,t)=\inf_{\xi^\nu\in C^1}
E\left[\int^t_0L^c(\gamma^\nu(s),s,\xi^\nu(\gamma^\nu(s),s))ds + v^0(\gamma^\nu(0)) \right]+h(c )t,
\end{eqnarray}
where $E$ stands for the expectation with respect to the Wiener measure  and $\gamma^\nu$ is a solution of the stochastic ODE
\begin{eqnarray}\label{S-ODE}
 d\gamma^\nu(s)=\xi^\nu(\gamma^\nu(s),s)ds+\sqrt{2\nu} dB(t-s),\,\,\,\, \gamma^\nu(t)=x.
 \end{eqnarray}
 Here $B$ is the standard Brownian motion. 
There exists the unique minimizing vector field $\xi^\nu{}^\ast$ of (\ref{stochastic-value-func}). The value $u^\nu(x,t)$ is given by 
\begin{eqnarray}\label{stochastic-value-func_x}
u^\nu(x,t)=
E\left[\int^t_0L^c_x(\gamma^\nu{}^\ast(s),s,\xi^\nu{}^\ast(\gamma^\nu{}^\ast(s),s))ds + u^0(\gamma^\nu{}^\ast(0)) \right],
\end{eqnarray} 
where $\gamma^\nu{}^\ast$ is a solution of (\ref{S-ODE}) with $\xi^\nu=\xi^\nu{}^\ast$. It is proved from a stochastic and variational point of view that, for $\nu\to0+$, $v^\nu$ converges uniformly to $v$ with the error $O(\sqrt{\nu})$ and $u^\nu$ converges pointwise to $u=v_x$ except for points of discontinuity of $u$. In particular, $u^\nu$ converges uniformly to $u$ without an arbitrarily small neighborhood of shocks.  The proof indicates how the stochastic variational formula (\ref{stochastic-value-func}) and   (\ref{stochastic-value-func_x}) tend to the deterministic ones (\ref{value-func}) and (\ref{value-func-CL}). Asymptotics of (\ref{S-ODE}) for $\nu\to0$ plays a central role. Fleming's approach yields much information and concrete pictures of the vanishing viscosity method. In particular we can see how the parabolicity disappears to be hyperbolic. 

The purpose of this paper is to establish a stochastic and variational approach to the finite difference method with the Lax-Friedrichs scheme. We discretize the equation of (\ref{CL}) by the Lax-Friedrichs scheme:
\begin{eqnarray}\label{CL-Delta}
 \frac{u^{k+1}_{m+1}-\frac{(u^k_{m}+u^k_{m+2})}{2}}{\dt}
+\frac{H(x_{m+2},t_k,c+u^k_{m+2})-H(x_m,t_k,c+u^k_m)}{2\dx}
=0.
\end{eqnarray}
We will see in the next section that we can find a difference equation which approximates the equation of (\ref{HJ}) and is equivalent to (\ref{CL-Delta}) in the sense that $u^k_m=(v^k_{m+1}-v^k_{m-1})/2\dx$:
\begin{eqnarray}\label{HJ-Delta}\frac{ v^{k+1}_{m} - \frac{(v^k_{m-1}+v^k_{m+1})}{2} }{\dt}+H(x_{m},t_k,c+\frac{v^k_{m+1}-v^k_{m-1}}{2\dx})
=h(c).
\end{eqnarray}
We present stochastic calculus of variations associated with (\ref{HJ-Delta}) that we minimize the expectation of a discrete action functional among  space-time inhomogeneous random walks in $\dx\Z\times\dt\Z$. This yields representation formulas of $v^k_{m+1}$ and  $u^k_m$ similar to (\ref{stochastic-value-func}) and (\ref{stochastic-value-func_x}). The probability measures of such random walks are no longer related to the Winer measure. This is the main difficulty of our arguments: We need to study the asymptotics for $\dx,\dt\to0$ of the random walks generated by arbitrary transition probabilities depending on space and time, under {\it hyperbolic scaling $0<\lambda_0\le\dt/\dx\le\lambda_1$.} We will see that the continuous limit of minimizing random walks is deterministic. In other words, we obtain the law of large numbers, where the minimizing random walks converge to the minimizing curves for  $u,v$. This proves convergence of the Lax-Friedrichs scheme. It is interesting to note that, under diffusive scaling $\dx^2/\dt=2\nu>0$,  the continuous limit of a certain class of random walks  is the Brownian motion or diffusion processes, and the solutions of (\ref{CL-Delta}) and (\ref{HJ-Delta}) converge to these of (\ref{CL-nu}) and (\ref{HJ-nu}). Our approach also yields much information and concrete pictures of the finite difference method with the Lax-Friedrichs scheme. In particular {\it we can see how the ``parabolicity'' due to numerical viscosity disappears to be hyperbolic in terms of the law of large numbers.} Here we point out several advantages due to  our approach: 
\begin{enumerate} 
 \item[(1)] Stability of the Lax-Friedrichs scheme and therefore convergence of the scheme up to arbitrary $T>0$ is derived from variational techniques.
 \item[(2)]The pointwise convergence of $u^k_m$ to $u=v_x$ is proved. In particular this yields the uniform convergence except neighborhoods of shocks with arbitrarily small measure. 
 \item[(3)] The uniform convergence of $v^k_{m+1}$ to $v$ with an error $O(\sqrt{\dx})$ is proved from the stochastic viewpoint. 
 \item[(4)] The approximation of (backward) characteristic curves of (\ref{CL}) and (\ref{HJ}) and its convergence are verified.         
\end{enumerate}

Approximation of entropy solutions with the Lax-Friedrichs (also with other schemes) is basically based on the $L^1$-frameworks with a priori estimates, where $\dx,\dt$-independent boundedness of both $u^k_m$ and its total variation must be verified. Our stochastic and variational approach is totally different from the $L^1$-frameworks  and proofs are simpler. 

In the case of flux functions which are independent of $x$ and $t$,  many details are known. Crandall and Majda \cite{Crandall-Majda} prove stability and $L^1$-convergence properties of monotone difference approximations in a rather general setting, where a flux function $H(p)$ is not necessarily convex.  Tadmor \cite{Tadmor1} shows the Lipschitz one-sided boundedness $(u^k_{m+2}-u^k_m)/2\dx\le \frac{a}{k\dt}$, which guarantees time-global stability.  \c{S}abac \cite{Sabac} proves that the optimal $L^1$-convergence rate of $u^k_m\to u$ is $O(\sqrt{\dx})$.   

In the case of flux functions which depend on both $x$ and $t$, the problem becomes much harder: Oleinik \cite{Oleinik} extensively investigate the Lax-Freidrichs scheme in this case, where stability and  $L^1$-convergence are proved with restricted $T>0$.
 This restriction is not satisfactory, because admissible $T$ must be less than $\int_{\norm u^0\norm_{L^\infty}}^\infty \frac{1}{V(r)}dr$ with 
 a monotone increasing $C^1$-function $V$ such that $\sup_{x,t\in\T,|p|\le r}|H_x(x,t,p)|\le V(r)$ ($V(r)=ar^2$ is often the case, due to $H_{pp}>0$).  Nishida and Soga \cite{Nishida-Soga} show the time-global stability of the Lax-Freidrichs scheme as well as the long time behavior that any difference solution converges exponentially to a periodic state as $k\to\infty$ yielding $\Z^2$-periodic entropy solutions. This argument is a generalization of results of Oleinik \cite{Oleinik} and Tadmor \cite{Tadmor1}.  However  they still assume that the flux function is of the form $H(x,t,p)=\frac{1}{2}p^2-F(x,t)$. It seems extremely hard to prove stability of the Lax-Friedrichs scheme for arbitrary $T>0$ with $H(x,t,p)$ satisfying (A1)-(A4), in a similar approach.    

There are also lots of works on approximation of viscosity solutions of Hamilton-Jacobi equations. Many of them are done independently of entropy solutions. We remark that, even in the case of (\ref{CL}) and (\ref{HJ}),  convergence results for viscosity solutions do not necessarily imply these for entropy solutions. We point out Crandall and Lions \cite{Crandall-Lions-approximation} and Souganidis \cite{Souganidis}, where convergence results for viscosity solutions with general schemes are established as well as error estimates of $O(\sqrt{\dx})$.  Our results provide interpretation of the order $\sqrt{\dx}$ from the stochastic viewpoint. 

Finally we refer to applications of our results to the weak KAM theory. The weak KAM theory makes clear the important connection among entropy solutions, viscosity solutions and the Hamiltonian dynamics generated by  $H(x,t,p)$. Approximation theories of the weak KAM theory should be developed with approximation methods which provide all of  entropy solutions, viscosity solutions and their characteristic curves at the same time, because the weak KAM theory requires three of them to connect the PDEs with important properties of the Hamiltonian dynamics. Fleming's results indicate that the vanishing viscosity method meets the requirement. Bessi \cite{Bessi} and other authors successfully exploits Fleming's approach to develop smooth approximation in the weak KAM theory. 
Nishida and Soga \cite{Nishida-Soga} develop difference approximation in the weak KAM theory by the Lax-Friedrichs scheme with the usual $L^1$-framework, where some arguments are not yet mathematically completed. 
Our results here indicate that the Lax-Friedrichs scheme also provides approximation to all of entropy solutions, viscosity solutions and their characteristic curves at the same time. The results of this paper can be strong tools for numerical analysis of  the weak KAM theory.   
    
\setcounter{section}{1}
\setcounter{equation}{0}
\section{Results}
\subsection{Equivalent schemes for conservation laws and Hamilton-Jacobi equations}
Let $N,K$ be natural numbers. The mesh size $\Delta=(\dx,\dt)$ is defined by $\dx:=(2N)^{-1}$ and $\dt:=(2K)^{-1}$. 
Set $\lambda:=\dt/\dx$, $x_m:=m\dx$ for $m\in\Z$ and $t_k:=k\dt$  for $k=0,1,2,\cdots$. For $x\in\R$ and $t>0$, the notation $m(x),k(t)$ denote the integers $m,k$ for which $x\in[x_m,x_m+2\dx),t\in[t_k,t_k+\dt)$. Let $(\dx\Z)\times(\dt\Z_{\ge0})$ be the set of all $(x_m,t_k)$ and $$\mathcal{G}_{even}\subset (\dx\Z)\times(\dt\Z_{\ge0}),\qquad\mathcal{G}_{odd}\subset (\dx\Z)\times(\dt\Z_{\ge0})$$ 
be the set of all $(x_m,t_k)$ with $k=0,1,2,\cdots$ and $m\in\Z$ with $m+k=$even, odd. We call $\mathcal{G}_{even}$, $\mathcal{G}_{odd}$ the even grid, odd grid.  
We consider the discretization of (\ref{CL}) by the Lax-Freidrichs scheme in $\mathcal{G}_{even}$:
\begin{eqnarray}\label{2CL-Delta}
\left\{
\begin{array}{lll}
&\dis \frac{u^{k+1}_{m+1}-\frac{(u^k_{m}+u^k_{m+2})}{2}}{\dt}
+\frac{H(x_{m+2},t_k,c+u^k_{m+2})-H(x_m,t_k,c+u^k_m)}{2\dx}
=0, \\\\\medskip\medskip
&u^0_m=u^0_\Delta(x_m),\quad u^k_{m\pm 2N}=u^k_{m}, 
\end{array}
\right.
\end{eqnarray}
where for $m=$ even 
\begin{eqnarray}\label{u^0}
u_\Delta^0(x):=\frac{1}{2\dx}\int^{x_{m}+\dx}_ {x_m-\dx}u^0(y)dy\mbox{\quad for $x\in[x_m-\dx,x_{m}+\dx)$}.
\end{eqnarray}
Note that $\dis\sum_{\{m\,|\,0\le m< 2N,\,m+k=even\}}u^k_m\cdot2\dx$ is conservative with respect to $k$ and is zero for $u^0$ with the average zero.   Now we consider a discrete version of (\ref{HJ}) in $\mathcal{G}_{odd}$:
\begin{eqnarray}\label{2HJ-Delta}\quad
\left\{
\begin{array}{lll}
&\dis \frac{ v^{k+1}_{m} - \frac{(v^k_{m-1}+v^k_{m+1})}{2} }{\dt}
+H(x_{m},t_k,c+\frac{v^k_{m+1}-v^k_{m-1}}{2\dx})
=h(c),\\\\\medskip\medskip
&v^0_{m+1}=v^0_\Delta(x_{m+1}),\quad v^k_{m+1\pm 2N}=v^k_{m+1},
\end{array}
\right.
\end{eqnarray}
where $v^0_\Delta$ is a function which converges to $v^0$ uniformly as $\Delta\to0$.   We introduce the difference operators: 
$$D_tw^{k+1}_m:=\frac{w^{k+1}_m-\frac{w^k_{m-1}+w^k_{m+1}}{2}}{\dt},\quad D_xw^{k}_{m+1}:=\frac{w^{k}_{m+1}-w^k_{m-1}}{2\dx}.$$   
In addition to the assumption $u^0=v^0_x$, we assume that  
\begin{eqnarray}\label{v^0}
v^0_\Delta(x):=v^0(0)+\int^x_0u^0_\Delta(y)dy.
\end{eqnarray}
Note that $u_\Delta^0\to u^0$ in $L^1$ and $v_\Delta^0\to v^0$ uniformly with $\norm v^0_\Delta-v^0\norm_{C^0}\le \norm u^0\norm_{L^\infty}\cdot 3\dx$, as $\Delta\to0$. The two problems (\ref{2CL-Delta}) and (\ref{2HJ-Delta}) are equivalent under (\ref{u^0}) and (\ref{v^0}): 
\begin{Prop}\label{Prop-CL-HJ}
 Let $u^k_m$ and $v^k_{m+1}$ be the solutions of (\ref{2CL-Delta}) and (\ref{2HJ-Delta}) with (\ref{u^0}) and  (\ref{v^0}). Then the one is derived from the other. In particular we have  $D_xv^k_{m+1}=u^k_m$.
\end{Prop}
\subsection{Random walks for the Lax-Friedrichs scheme}
We introduce space-time inhomogeneous random walks in $\mathcal{G}_{odd}$, which corresponds to the stochastic ODE (\ref{S-ODE}) in the vanishing viscosity method or characteristic curves of (\ref{CL}). 

For each point  $(x_n,t_{l+1})\in\mathcal{G}_{odd}$, we consider backward random walks $\gamma$ which start from $x_n$ at $t_{l+1}$ and move by $\pm\dx$ in each backward time step: 
$$\gamma=\{\gamma^k\}_{k=0,1,\cdots,l+1},\quad\gamma^{l+1}=x_{n},\quad \gamma^{k+1}-\gamma^k=\pm\dx.$$
More precisely, we set the following for each $(x_n,t_{l+1})\in\mathcal{G}_{odd}$: 
\begin{eqnarray*}
&&X^k:=\{ x_m \,|\, \mbox{ $(x_m,t_k)\in\mathcal{G}_{odd}$, $|x_m-x_n|\le(l+1-k)\dx$}\},\,\,\,k\le l+1\\
&&G:=\bigcup_{1\le k\le l+1}\big(X^{k}\times\{t_{k}\}\big)\subset\mathcal{G}_{odd}, \\
&&\xi:G\ni(x_m,t_k)\mapsto\xi^k_m\in[-\lambda^{-1},\lambda^{-1}],\quad \lambda=\dt/\dx, \\
&&\bar{\bar{\rho}}: G\ni(x_m,t_k)\mapsto\bar{\bar{\rho}}^k_m:=\frac{1}{2}-\frac{1}{2}\lambda\xi^k_m\in[0,1],\\
&&\bar{\rho}: G\ni(x_m,t_k)\mapsto\bar{\rho}^k_m:=\frac{1}{2}+\frac{1}{2}\lambda\xi^k_m\in[0,1],\\
&&\gamma:\{ 0,1,2,\cdots,l+1\}\ni k\mapsto \gamma^k\in X^k,\mbox{ $\gamma^{l+1}=x_n$, $\gamma^{k+1}-\gamma^k=\pm\dx$  },\\
&&\Omega:\mbox{ the family of $\gamma$}. 
\end{eqnarray*}
We regard $\bar{\bar{\rho}}^k_m$ (respectively $\bar{\rho}^k_m$) as a transition probability from $(x_m,t_k)$ to $(x_m+\dx,t_k-\dt)$ (from $(x_m,t_k)$  to $(x_m-\dx,t_k-\dt)$).  We control the transition of our random walks by $\xi$, which plays a ``velocity field''-like role in $G$.     
We define the density of each path $\gamma\in\Omega$ as 
$$\mu(\gamma):=\prod_{1\le k\le l+1}\rho(\gamma^{k},\gamma^{k-1}),$$    
where $\rho(\gamma^{k},\gamma^{k-1})=\bar{\bar{\rho}}^k_{m(\gamma^{k})}$ (respectively $\bar{\rho}^k_{m(\gamma^{k})}$) if $\gamma^{k}-\gamma^{k-1}=-\dx$ ($\dx$).  The density $\mu(\cdot)=\mu(\cdot;\xi)$ yields a probability measure of $\Omega$, namely 
$$prob(A)=\sum_{\gamma\in A}\mu(\gamma;\xi)\mbox{\quad for $A\subset\Omega$}. $$
The expectation with respect to this probability measure is denoted by $E_{\mu(\cdot;\xi)}$, namely for a random variable $f:\Omega\to\R$
$$E_{\mu(\cdot;\xi)}[f(\gamma)]:=\sum_{\gamma\in\Omega}\mu(\gamma;\xi)f(\gamma).$$
The above objects depend on the initial point $(x_n,t_{l+1})$, but we omit to add the index to them for simpler notation.   We remark that, since our transition probabilities are space-time inhomogeneous, the well-known law of large numbers and central limit theorem for random walks do not always hold in our case. Soga {\cite{Soga2}} investigate the asymptotics for $\Delta\to0$ of the probability measure of $\Omega$ under hyperbolic scaling, which will be  important in this work.    
\subsection{Stochastic and Variational Representation of Approximate Solutions}
We represent the approximate solutions by a discrete action functional with the random walks and $L^c=L-c\xi$. From now on we assume the following: 
\medskip
 
 \noindent{\bf Assumption.} {\it Suppose (A1)-(A4). Let $T>0$ be arbitrarily fixed. The parameter $c$ varies within $[c_0,c_1]$. Initial datas are bounded: $\norm u^0\norm_{L^\infty}=\norm v^0_x\norm_{L^\infty}\le r,\,\,\,\norm v^0 \norm_{C^0}\le r$. }
 \medskip

 First of all we see a result, assuming also that there exists a solution $u^k_m$ of  (\ref{2CL-Delta}) which satisfies the stability condition called the CFL-condition
$$|H_p(x_m,t_k,c+u^k_{m})|<\lambda^{-1}\quad (\lambda=\dt/\dx).$$
\begin{Prop}\label{Main1} 
Suppose that we have the solution $v^k_m$ of (\ref{2HJ-Delta}) for which $u^k_m:=D_xv^{k}_{m+1}$ satisfies the CFL-condition for all $m$ and $k=0,1,2,\cdots,k^\ast$. 
Then the solution is represented for each $n$ and $0<l+1\le k^\ast+1$ as 
\begin{eqnarray}\label{value-func-Delta}
v^{l+1}_n=\inf_{\xi} E_{\mu(\cdot;\xi)}\Big[\sum_{0<k\le l+1}L^c(\gamma^k,t_{k-1},\xi^k_{m(\gamma^k)})\dt +v^0_\Delta(\gamma^0)\Big]+h(c)t_{l+1}.
\end{eqnarray}
 The minimizing velocity field $\xi^\ast$ exists, which is unique and given by 
 $$\xi^\ast{}^{k+1}_m=H_p(x_m,t_{k},c+D_xv^{k}_{m+1}).$$
\end{Prop}
\noindent This proposition is informative, because a proof indicates how the Lax-Friedrichs scheme reveals the stochastic and variational structure. The proof also implies that Proposition \ref{Main1} holds only with the assumptions (A2) and (A3).   

Next we remove the assumption of the existence of $v^k_{m+1}$ satisfying the CFL-condition:
\begin{Thm}\label{Main2} 
There exists $\lambda_1>0$ (depending  on $T$, $[c_0,c_1]$ and $r$, but independent of $\Delta$) for which we have the following:
\begin{enumerate}
\item For any small $\Delta=(\dx,\dt)$ with $\lambda=\dt/\dx<\lambda_1$, the expectation of the following functional with each $n$ and $0<l+1\le k(T)$
\begin{eqnarray*}\label{action-delta}
E^{l+1}_n(\xi):=E_{\mu(\cdot;\xi)}\Big[\sum_{0<k\le l+1}L^c(\gamma^k,t_{k-1},\xi^k_{m(\gamma^k)})\dt +v^0_\Delta(\gamma^0)\Big]+h(c)t_{l+1}
\end{eqnarray*}
has the infimum denoted by $V^{l+1}_n$ with respect to $\xi:G\to[-\lambda^{-1},\lambda^{-1}]$.   The infimum $V^{l+1}_n$ is attained by $\xi^\ast $ which satisfies $|\xi^\ast|\le\lambda_1^{-1}<\lambda^{-1}$. 
\item Define $v^{k}_{m+1}$ for each  $m$ and $0\le k\le  k(T)$ as 
$v^0_{m+1}:=v^0_\Delta(x_{m+1})$, $v^{k}_{m+1}:=V^k_{m+1}$. 
Then the minimizing velocity field $\xi^\ast$ for each $V^{l+1}_n$ satisfies in $G$ 
$$L^c_\xi(x_m,t_{k},\xi^\ast{}^{k+1}_m)=D_xv^{k}_{m+1}\Leftrightarrow 
\xi^\ast{}^{k+1}_m=H_p(x_m,t_{k},c+D_xv^{k}_{m+1}).$$  
\item The above $v^{k}_{m+1}$ satisfies (\ref{2HJ-Delta}) for $0\le k \le k(T) $.
\end{enumerate}
\end{Thm}
\noindent This theorem immediately leads to one of our main results:
 \begin{Thm}\label{Main3} 
 There exists $\lambda_1>0$ (depending  on $T$, $[c_0,c_1]$ and $r$, but independent of $\Delta$) such that for any small $\Delta=(\dx,\dt)$ with $\lambda=\dt/\dx<\lambda_1$  we have the solution $u^k_m$ of (\ref{2CL-Delta}) which is bounded and satisfies the CFL-condition up to $k=k(T)$:
$$|H_p(x_m,t_{k},c+u^{k}_{m})|\le\lambda_1^{-1}< \lambda^{-1}\mbox{.}$$ 
 \end{Thm}
\noindent Next we ``represent'' the solution $u^k_m$ of (\ref{2CL-Delta}):
\begin{Thm}\label{Main4}
For each $n$ and $0<l+1\le k(T)$, let $\xi^\ast$ be the minimizer for $V^{l+1}_n$ and  $\gamma,\mu(\cdot;\xi^\ast)$ be the minimizing random walk for $V^{l+1}_n$. Let $\tilde{\xi}^\ast$ be the minimizer for $V^{l+1}_{n+2}$ and  $\tilde{\gamma},\tilde{\mu}(\cdot;\tilde{\xi}^\ast)$ be the minimizing random walk for $V^{l+1}_{n+2}$. Then $u^{l+1}_{n+1}$ satisfies 
\begin{eqnarray*}\label{5value-entropy}
u^{l+1}_{n+1}&\le& E_{\mu(\cdot;\xi^\ast)}\Big[\sum_{0<k\le l+1}L^c_x(\gamma^k,t_{k-1},\xi^\ast{}^k_{m(\gamma^k)})\dt+u^0_\Delta(\gamma^0+\dx) \Big]+O(\dx),
\\ \label{5value-entropy2}
u^{l+1}_{n+1}&\ge& E_{\tilde{\mu}(\cdot;\tilde{\xi}^\ast)}\Big[\sum_{0<k\le l+1}L^c_x(\tilde{\gamma}^k,t_{k-1},\tilde{\xi}^\ast{}^k_{m(\tilde{\gamma}^k)})\dt
+u^0_\Delta(\tilde{\gamma}^0-\dx) \Big]+O(\dx),
\end{eqnarray*}
where $O(\dx)$ stands for a number of $(-\theta\dx,\theta\dx)$ with $\theta>0$ independent of $\dx$.
\end{Thm}
\subsection{Convergence of Approximation}
We present convergence results of the stochastic and variational approach to the Lax-Friedrichs scheme.  We always take the limit $\Delta=(\dx,\dt)\to0$ under hyperbolic scaling $0<\lambda_0\le\lambda=\dt/\dx<\lambda_1$.  We say that a point $(x,t)\in\T\times(0,T]$ is a regular point, if there exists $v_x(x,t)$. Note that regular points are nothing but  points of continuity of $u=v_x$  and almost every points are regular.  The minimizing curve of $v(x,t)$ is unique, if $(x,t)$ is regular (see e.g. \cite{Cannarsa}). 
\begin{Thm}\label{Main5}
Let $v$ be the viscosity solution of (\ref{HJ}) and $v_\Delta$ be the linear interpolation of the solution $v^k_{m+1}$ of (\ref{2HJ-Delta}). Then 
$$\mbox{$v_\Delta\to v$ uniformly on $\T\times[0,T]$ as $\Delta\to 0$.}$$
In particular, we have an error estimate: There exists $\beta>0$ independent of $\Delta$ such that  
$$\norm v_\Delta-v\norm_{C^0(\T\times[0,T])}\le \beta\sqrt{\dx}.$$ 
\end{Thm}
\noindent This result is consistent with the earlier literature cited in Introduction. The argument is based on the different viewpoint that the random walks become deterministic and our stochastic calculus of variations tend to the deterministic ones as $\Delta\to0$.
\begin{Thm}\label{Main6}
Let $(x,t)\in\T\times(0,T]$ be a regular point and $\gamma^\ast:[0,t]\to\R$ be the minimizing curve for $v(x,t)$. Let $(x_n,t_{l+1})$ be a point of $[x-2\dx,x+2\dx)\times[t-\dt,t+\dt)$ and $\gamma_\Delta:[0,t]\to\R$ be the linear interpolation of the random walk $\gamma$ generated by the minimizing velocity field $\xi^\ast$ for $v^{l+1}_n$. Then
$$\gamma_\Delta\to\gamma^\ast\mbox{ uniformly on $[0,t]$ in probability as  $\Delta\to0$}.$$
In particular, the average of $\gamma_\Delta$ converges uniformly to $\gamma^\ast$ as $\Delta\to0$. 
\end{Thm}
\noindent The minimizing curve $\gamma^\ast$ is the genuine backward characteristic curves of $u,v$ starting from $(x,t)$. Therefore the Lax-Friedrichs scheme turns out to approximate not only PDE solutions but also their characteristic curves at the same time.       
\begin{Thm}\label{Main7}
Let $u=v_x$ be the entropy solution of (\ref{CL}) and $u_\Delta$ be the step function derived from the solution $u^k_m$ of (\ref{2CL-Delta}), namely  $u_\Delta(x,t)=u^k_m$ for $(x,t)\in[x_m-\dx,x_m+\dx)\times[t_k,t_k+\dt)$.  Then for each regular point $(x,t)\in\T\times[0,T]$
$$u_\Delta(x,t)\to u(x,t)\mbox{ as $\Delta\to0$}.$$
In particular, $u_\Delta$ converges uniformly to $u$ on $(\T\times[0,T])\setminus\Theta$, where $\Theta$ is a neighborhood of the set of points of singularity of $u$  with  an arbitrarily small measure.    
\end{Thm}
\noindent This convergence result is stronger than the one derived from the usual $L^1$-framework in the following sense: The approximate solution $u_\Delta$ converges pointwise a.e. and therefore ``uniformly'' to the particular representative element of the $L^1$-valued entropy solution, which is the derivative of the corresponding viscosity solution and is represented as (\ref{value-func-CL}). 
\setcounter{section}{2}
\setcounter{equation}{0}
\section{Proof of Results}
We prove the propositions and theorems in Section 2. 
\begin{proof}[{\bf Proof of Proposition \ref{Prop-CL-HJ}.}] 
It follows from (\ref{u^0}) and (\ref{v^0}) that $D_xv^0_{m+1}=u^0_m$. Subtracting $D_tv^{k+1}_{m}+H(x_{m},t_k,D_xv^k_{m+1})=h(c)$ from $D_tv^{k+1}_{m+2}+H(x_{m+2},t_k,D_xv^k_{m+3})=h(c)$ and dividing it by $2\dx$, we see that $D_xv^k_{m+1}$ is equal to $u^k_m$. 
 
 Converse argument: We inductively define $\tilde{v}^k_{m+1}$ in $\mathcal{G}_{odd}$ as
$$\tilde{v}^k_{m+1}=\tilde{v}^k_{m-1}+u^k_m\cdot2\dx,\quad
\tilde{v}^k_{-(m+1)}=\tilde{v}^k_{-(m-1)}+u^k_{-m}\cdot(-2\dx)$$
with $\tilde{v}^k_{1}:=u^k_0\dx$ for $k=$ even and $\tilde{v}^k_{0}:=0$ for $k=$ odd. Then $\tilde{v}^k_{m+1}$ satisfies $D_x\tilde{v}^k_{m+1}=u^k_{m}$, $\tilde{v}^k_{m+1\pm 2N}=\tilde{v}^k_{m+1}$ and  $\tilde{v}^0_{m+1}=v_\Delta^0(x_{m+1})-v^0(0)$. Using the equation of (\ref{2CL-Delta}) and the equalities
\begin{eqnarray*}
&&D_t\tilde{v}^{k+1}_m=D_t\tilde{v}^{k+1}_0+\left\{ D_tu^{k+1}_1+D_tu^{k+1}_3+\cdots+D_tu^{k+1}_{m-1}  \right\}2\dx\mbox{\quad for $k=even$},\\
&&D_t\tilde{v}^{k+1}_m=D_t\tilde{v}^{k+1}_1+\left\{ D_tu^{k+1}_2+D_tu^{k+1}_4+\cdots+D_tu^{k+1}_{m-1}  \right\}2\dx\mbox{\quad for $k=odd$}
\end{eqnarray*}
(note that $D_t\tilde{v}^{k+1}_0=0$), we obtain 
\begin{eqnarray*}
D_t\tilde{v}^{k+1}_{m} +H(x_{m},t_k,c+D_x\tilde{v}^k_{m+1})=P^k,
\end{eqnarray*}
where for $k=\mbox{even, odd}$
\begin{eqnarray*}
P^k= D_t\tilde{v}^{k+1}_{0}+H(x_{0},t_k,c+D_x\tilde{v}^k_{1}),
\quad P^k=D_t \tilde{v}^{k+1}_{1}+H(x_{1},t_k,c+D_x\tilde{v}^k_{2}).
\end{eqnarray*}
We define $v^k_{m+1}$ as 
$$v^0_{m+1}:=\tilde{v}^0_{m+1}+v^0(0),\quad v^k_{m+1}:=\tilde{v}^k_{m+1}-\sum_{0\le k'<k }(P^{k'}-h(c))\dt.$$
Then $v^k_{m+1}$ satisfies (\ref{2HJ-Delta}) and $D_xv^k_{m+1}=u^k_{m}$. 
\end{proof}
Define $\Omega^k_m:=\{\gamma\in\Omega\,|\, \gamma^k=x_m\}$ and $p^k_m:=\sum_{\gamma\in\Omega^k_m}\mu(\gamma)$. We observe the following lemma, which follows from the definition of random walks. 
\begin{Lemma}\label{Lemma1}
 \begin{enumerate}
 \item $\dis\sum_{x\in X^k} p^k_{m(x)}=1$. Hence $\{p^k_{m(x)}\}_{x\in X^k}$ yields a probability of $X^k$.
\item $\dis p^{k}_m=p^{k+1}_{m-1}\bbrho^{k+1}_{m-1}+p^{k+1}_{m+1}\brho^{k+1}_{m+1}$, where $\bbrho^{k+1}_{m\pm1},\brho^{k+1}_{m\pm1}=0$ if $x_{m\pm1}\not\in\ X^{k+1}$.
 \end{enumerate}
\end{Lemma}
\begin{proof}[{\bf Proof of Proposition \ref{Main1}.}]
Fix $\xi:G\to[-\lambda^{-1},\lambda^{-1}]$ arbitrarily. It follows form the difference equation of (\ref{2HJ-Delta}) and the property of the Legendre transform that
\begin{eqnarray*}
v^{l+1}_n&=&\frac{v^l_{n-1}+v^l_{n+1}}{2}-H(x_n,t_l,c+D_xv^l_{n+1})\dt+h(c)\dt\\
&=& \{\xi^{l+1}_n\cdot(c+D_xv^l_{n+1})-H(x_n,t_l,c+D_xv^l_{n+1})\}\dt-c\xi^{l+1}_n\dt\\
&&+(\frac{1}{2}+\frac{1}{2}\lambda\xi^{l+1}_n)v^l_{n-1}+(\frac{1}{2}-\frac{1}{2}\lambda\xi^{l+1}_n)v^l_{n+1}+h(c)\dt\\
&\le&L^c(x_n,t_{l},\xi^{l+1}_n)\dt+(\frac{1}{2}+\frac{1}{2}\lambda\xi^{l+1}_n)v^l_{n-1}+(\frac{1}{2}-\frac{1}{2}\lambda\xi^{l+1}_n)v^l_{n+1}+h(c)\dt,
\end{eqnarray*}
where the equality holds, if and only if  $\xi^{l+1}_n=H_p(x_n,t_{l},c+D_xv^{l}_{n+1})\in (-\lambda^{-1},\lambda^{-1})$. Similarly we have 
\begin{eqnarray*}
v^l_{n-1}
&\le&L^c(x_{n-1},t_{l-1},\xi^{l}_{n-1})\dt
+(\frac{1}{2}+\frac{1}{2}\lambda\xi^{l}_{n-1})v^{l-1}_{n-2}+(\frac{1}{2}-\frac{1}{2}\lambda\xi^{l}_{n-1})v^{l-1}_{n}+h(c)\dt,\\
v^l_{n+1}
&\le&L^c(x_{n+1},t_{l-1},\xi^{l}_{n+1})\dt
+(\frac{1}{2}+\frac{1}{2}\lambda\xi^{l}_{n+1})v^{l-1}_{n}+(\frac{1}{2}-\frac{1}{2}\lambda\xi^{l}_{n+1})v^{l-1}_{n+2}+h(c)\dt,
\end{eqnarray*} 
where the equality holds, if and only if  $\xi^{l}_{n\pm1}=H_p(x_{n\pm1},t_{l-1},c+D_xv^{l-1}_{n\pm1+1})\in (-\lambda^{-1},\lambda^{-1})$. Hence, by Lemma \ref{Lemma1},  we obtain
\begin{eqnarray*} 
v^{l+1}_n\le\sum_{l\le k\le l+1}\left(\sum_{x\in X^k}p^k_{m(x)}L^c(x,t_{k-1},\xi^{k}_{m(x)}) \right)\dt+\sum_{x\in X^{l-1}}p^{l-1}_{m(x)}v^{l-1}_{m(x)} +h(c)(t_{l+1}-t_{l-1}). 
\end{eqnarray*}
Continuing this process, we obtain
\begin{eqnarray*} 
v^{l+1}_n\le\sum_{0< k\le l+1}\left(\sum_{x\in X^k}p^k_{m(x)}L^c(x,t_{k-1},\xi^{k}_{m(x)}) \right)\dt+\sum_{x\in X^{0}}p^{0}_{m(x)}v^0_{m(x)} +h(c)t_{l+1}.
\end{eqnarray*}
The equality holds, if and only if  $\xi^k_m=H_p(x_m,t_{k-1},c+D_xv^{k-1}_{m+1})\in (-\lambda^{-1},\lambda^{-1})$.  We see that the first and second term of the right hand side, denoted by $A_1$ and $A_2$, are changed into  
\begin{eqnarray*} 
A_1&=&\sum_{0< k\le l+1}\Big\{
\sum_{x\in X^k}\Big(\sum_{\gamma\in\Omega^k_{m(x)}}\mu(\gamma;\xi)\Big)
L^c(\gamma^k,t_{k-1},\xi^{k}_{m(\gamma^k)}) \Big\}\dt\\
&=&\sum_{0< k\le l+1}\Big(
\sum_{\gamma\in\Omega}\mu(\gamma;\xi)
L^c(\gamma^k,t_{k-1},\xi^{k}_{m(\gamma^k)}) \Big)\dt\\
&=&\sum_{\gamma\in\Omega}\mu(\gamma;\xi)\Big(
\sum_{0< k\le l+1}
L^c(\gamma^k,t_{k-1},\xi^{k}_{m(\gamma^k)})\dt \Big),\\
A_2&=&\sum_{x\in X^{0}}\Big(\sum_{\gamma\in\Omega^0_{m(x)}}\mu(\gamma;\xi)\Big)v^0_{m(\gamma^0)}=\sum_{\gamma\in\Omega}\mu(\gamma;\xi)v^0_{m(\gamma^0)}.
\end{eqnarray*}
$\xi$ is arbitrary and we conclude (\ref{value-func-Delta}). 
\end{proof}
For simplicity, we introduce the following notation: $\xi|_{\gamma}:k\mapsto\xi^k_{m(\gamma^k)}$,
\begin{eqnarray*}
\mathcal{L}(\gamma,\eta)&:=&\int^t_0L^c(\gamma(s),s,\eta(s))ds+v^0(\gamma(0)):AC([0,t])\times L^1([0,t])\to\R,\\
\mathcal{L}_\Delta(\gamma,\eta)&:=&\sum_{0<k\le l+1}L^c(\gamma^k,t_{k-1},\eta^k)\dt+v^0_\Delta(\gamma^0):D\times D\to\R,
\end{eqnarray*}
where $D$ is the set of functions: $\{0,1,2,\cdots,l+1\}\to \R.$ 
\begin{proof}[{\bf Proof of Theorem \ref{Main2}.}]
It follows from (A1)-(A4) that  there exists $\alpha_1$ for which we have $|L^c_x|\le\alpha_1(|L^c|+1)$ for any $c\in[c_0,c_1]$ and $\dis L_\ast:=|\min\{0,\inf_{x,t,\xi,c}L^c\}|$ is bounded.  
Due to the lower boundedness of $L^c$, $E^{l+1}_n(\xi)$ has the infimum $V^{l+1}_n$ with respect to $\xi:G\to[-\lambda^{-1},\lambda^{-1}]$ for any $\Delta$. Note that, if $\xi^\ast$ is a minimizer for $V^{l+1}_n$, then for each $(x_m,t_k)\in G$ $V^k_m$ is attained by $\tilde{\xi}^\ast:\tilde{G}\subset G\to[-\lambda^{-1},\lambda^{-1}]$ which is the restriction of $\xi^\ast$ to $\tilde{G}$ whose vertex is $(x_m,t_k)$. Introduce  
\begin{eqnarray}\nonumber
&&\alpha_2:=r+T\max_{x,t}L(x,t,0),\,\,\,\,\,\alpha_3:=\alpha_1\{(1+2L_\ast)T+\alpha_2+r\},\\\label{lambda_1}
&&\lambda_1:=\max\{\max_{x,t}|H_p(x,t,c_1+1+r+\alpha_3)|,\,\,\max_{x,t}|H_p(x,t,c_0-1-r-\alpha_3)|\},\\\nonumber
&&\theta:=T\max_{x,t,|\xi|\le\lambda_1^{-1}} L_{xx}(x,t,\xi).
\end{eqnarray}

Now we assume that, for each $n$ and some $l$ such that $0\le l<k(T)$, the minimizer $\xi^\ast$ for $V^{l+1}_n$ satisfies $|\xi^\ast|\le\lambda_1^{-1}$. This is true for $l=0$, because of the boundedness $|D_xv^0_{m+1}|\le r$ and Proposition \ref{Main1}.  We want to prove the following estimate under $\dx\theta\le 1$:
\begin{eqnarray}\label{bound}
|D_xV^{l+1}_{n+2}|=\left|\frac{V^{l+1}_{n+2}-V^{l+1}_n}{2\dx}\right|\le 1+r+\alpha_3.
\end{eqnarray}
Before proving, we observe the result from (\ref{bound}): Consider $E^{l+2}_{n+1}(\xi)$ with $\xi$ whose restriction to $k\le l+1$ coincides with the minimizers for $V^{l+1}_n$ and $V^{l+1}_{n+2}$. Then, using the property of the Legendre transform,  we have for any $\xi^{l+2}_{n+1}$
\begin{eqnarray}\label{proof-3}
E^{l+2}_{n+1}(\xi)&=&L^c(x_{n+1},t_{l+1},\xi^{l+2}_{n+1})\dt+\brho^{l+2}_{n+1} V^{l+1}_{n}+\bbrho^{l+2}_{n+1} V^{l+1}_{n+2}+h(c)\dt\\\nonumber
&\ge&\{\xi^{l+2}_{n+1}\cdot(c+D_xV^{l+1}_{n+2})-H(x_{n+1},t_{l+1},c+D_xV^{l+1}_{n+2})\}\dt-c\xi^{l+2}_{n+1}\dt\\\nonumber
&&+(\frac{1}{2}+\frac{\lambda}{2}\xi^{l+2}_{n+1}) V^{l+1}_{n}+(\frac{1}{2}-\frac{\lambda}{2}\xi^{l+2}_{n+1}) V^{l+1}_{n+2}+h(c)\dt\\\nonumber
&=&\frac{V^{l+1}_{n}+V^{l+1}_{n+2}}{2}-H(x_{n+1},t_{l+1},c+D_xV^{l+1}_{n+2})\dt+h(c)\dt,
\end{eqnarray}
where the equality holds, if and only if $\xi^{l+2}_{n+1}=H_p(x_{n+1},t_{l+1},c+D_xV^{l+1}_{n+2})$. Therefore the minimizer $\xi^\ast$ for $V^{l+2}_{n+1}$ exists and satisfies 
\begin{eqnarray*}\label{proof-2}
\xi^\ast{}^{l+2}_{n+1}=H_p(x_{n+1},t_{l+1},c+D_xV^{l+1}_{n+2}).
\end{eqnarray*}
In particular, we have $|\xi^\ast{}^{l+2}_{n+1}|\le\lambda^{-1}_1$ and $|\xi^\ast|\le\lambda^{-1}_1$ because of (\ref{bound}). Thus, by induction, we complete the proof of 1. and 2.

We prove (\ref{bound}): Let $\xi^\ast$ be the minimizer for $V^{l+1}_n$ and $\mu,\gamma,\Omega$, etc. be the notation of the random walks starting from $(x_n,t_{l+1})$. Let $\tilde{\xi}^\ast$ be the minimizer for $V^{l+1}_{n+2}$ and  $\tilde{\mu},\tilde{\gamma},\tilde{\Omega}$, etc. be the notation of the random walks starting from $(x_{n+2},t_{l+1})$. We take $\mu(\cdot;\zeta)$ with $\zeta(x_m,t_k):=\tilde{\xi}^\ast(x_m+2\dx,t_k)$. Since $V^{l+1}_n$ is the infimum, we have  
 \begin{eqnarray*}
V^{l+1}_{n}\le E_{\mu(\cdot;\zeta)}\Big[\mathcal{L}_\Delta(\gamma,\zeta|_\gamma)\Big]+h(c)t_{l+1},
\quad V^{l+1}_{n+2}= E_{\tilde{\mu}(\cdot;\tilde{\xi}^\ast)}\Big[\mathcal{L}_\Delta(\tilde{\gamma},\tilde{\xi}^\ast|_{\tilde{\gamma}})\Big]+h(c)t_{l+1}.
\end{eqnarray*}
Since  $\Omega=\{\tilde{\gamma}-2\dx\,|\,\tilde{\gamma}\in\tilde{\Omega}\}$ and $\mu(\tilde{\gamma}-2\dx;\zeta)=\tilde{\mu}(\tilde{\gamma};\tilde{\xi}^\ast)$, we have
\begin{eqnarray*}
V^{l+1}_{n}&\le& E_{\mu(\cdot;\zeta)}\Big[\mathcal{L}_\Delta(\gamma,\zeta|_\gamma)\Big]+h(c)t_{l+1}\\
&=&\sum_{\tilde{\gamma}\in{\tilde{\Omega}}}
\tilde{\mu}(\tilde{\gamma};\tilde{\xi}^\ast)\Big[\sum_{0< k\le l+1}L^c(\tilde{\gamma}^k-2\dx,t_{k-1},\tilde{\xi}^\ast{}^k_{m(\tilde{\gamma}^k)})\dt+v^0_{m(\tilde{\gamma}^0-2\dx)}  \Big]+h(c)t_{l+1}.
\end{eqnarray*}
Therefore we obtain
\begin{eqnarray}\label{proof-Main5}
D_xV^{l+1}_{n+2}
&\ge& \frac{E_{\tilde{\mu}(\cdot;\tilde{\xi}^\ast)}\Big[\mathcal{L}_\Delta(\tilde{\gamma},\tilde{\xi}^\ast|_{\tilde{\gamma}})\Big]-E_{\mu(\cdot;\zeta)}\Big[\mathcal{L}_\Delta(\gamma,\zeta|_\gamma)\Big]}{2\dx}\\\nonumber
&=&E_{\tilde{\mu}(\cdot;\tilde{\xi}^\ast)}\Big[\sum_{0< k\le l+1}\frac{1}{2\dx} 
\{L^c(\tilde{\gamma}^k,t_{k-1},\tilde{\xi}^\ast{}^k_{m(\tilde{\gamma}^k)})\\\nonumber
&&-L^c(\tilde{\gamma}^k-2\dx,t_{k-1},\tilde{\xi}^\ast{}^k_{m(\tilde{\gamma}^k)}) \} \dt +\frac{v^0_{m(\tilde{\gamma}^0)}-v^0_{m(\tilde{\gamma}^0-2\dx)}}{2\dx}\Big]\\\nonumber
&=&E_{\tilde{\mu}(\cdot;\tilde{\xi}^\ast)}\Big[ \sum_{0< k\le l+1}
L^c_x(\tilde{\gamma}^k,t_{k-1},\tilde{\xi}^\ast{}^k_{m(\tilde{\gamma}^k)})\dt 
+u^0_{m(\tilde{\gamma}^0-\dx)}\Big]+O(\dx).
\end{eqnarray}
Using (A4) and noting that $\tilde{\xi}^\ast$ is the minimizer, we see that the fourth line of (\ref{proof-Main5}) denoted by $A$ is bounded from the  below by $-(1+r+\alpha_3)$: 
\begin{eqnarray*}
A&\ge& -E_{\tilde{\mu}(\cdot;\tilde{\xi}^\ast)}\Big[ \sum_{0< k\le l+1}
\alpha_1\big(1+|L^c(\tilde{\gamma}^k,t_{k-1},\tilde{\xi}^\ast{}^k_{m(\tilde{\gamma}^k)})|\big)\dt\Big]-r-\theta\dx\\
&=&-E_{\tilde{\mu}(\cdot;\tilde{\xi}^\ast)}\Big[ \sum_{0< k\le l+1}
\alpha_1\big(1+|L^c(\tilde{\gamma}^k,t_{k-1},\tilde{\xi}^\ast{}^k_{m(\tilde{\gamma}^k)})+L_\ast-L_\ast|\big)\dt\Big]-r-\theta\dx\\
&\ge&-E_{\tilde{\mu}(\cdot;\tilde{\xi}^\ast)}\Big[ \sum_{0< k\le l+1}
\alpha_1\big(1+L^c(\tilde{\gamma}^k,t_{k-1},\tilde{\xi}^\ast{}^k_{m(\tilde{\gamma}^k)})+L_\ast\big)\dt\Big]-\alpha_1L_\ast T-r-1\\
&=&-\alpha_1E_{\tilde{\mu}(\cdot;\tilde{\xi}^\ast)}\Big[\mathcal{L}_\Delta(\tilde{\gamma},\tilde{\xi}^\ast|_{\tilde{\gamma}})\Big]+\alpha_1 E_{\tilde{\mu}(\cdot;\tilde{\xi}^\ast)}\Big[v^0_\Delta(\tilde{\gamma}^0)\Big]-\alpha_1T-2\alpha_1L_\ast T-r-1\\
&\ge&-\alpha_1E_{\tilde{\mu}(\cdot;0)}\Big[\mathcal{L}_\Delta(\tilde{\gamma},0|_{\tilde{\gamma}})\Big]-\alpha_1r-\alpha_1T-2\alpha_1L_\ast T-r-1\\
&\ge&-\alpha_1\alpha_2-\alpha_1r-\alpha_1T-2\alpha_1L_\ast T-r-1\\
&=&-(\alpha_3+r+1).
\end{eqnarray*} 
\indent Similar reasoning by $\tilde{\mu}(\cdot;\tilde{\zeta})$ with $\tilde{\zeta}(x_m,t_k):=\xi^\ast(x_m-2\dx,t_k)$ yields the upper bound of $D_xV^{l+1}_{n+2}$. 

Since (\ref{proof-3}) becomes an equality for the minimizing velocity field, we conclude 3. 
\end{proof}
Theorem \ref{Main3} follows from Proposition \ref{Prop-CL-HJ} and Theorem \ref{Main2}. Theorem \ref{Main4} follows from (\ref{proof-Main5}) and the similar reasoning for the upper bound of $D_xV^{l+1}_{n+2}$. 

We observe several properties of the function $v:\T\times[0,T]\to\R$ defined as 
$$v(x,t):=\inf_{\gamma\in AC,\,\,\gamma(t)=x}\mathcal{L}(\gamma,\gamma')+h(c )t,\,\,\,v(\cdot,0):=v^0(\cdot),$$
which is the viscosity solution of (\ref{HJ}). There exists at least one minimizing curve for $v(x,t)$, due to Tonelli's theory. 
\begin{Lemma}\label{lemma-regularity}
Let $\gamma^\ast:[0,t]\to\R$ be a minimizing curve for $v(x,t)$.
 \begin{enumerate}
 \item The following regularity properties hold:
 \begin{eqnarray*}
&&L^c_{\xi}(\gamma^\ast(\tau),\tau,\gamma^\ast{}'(\tau))\in\p^-_xv(\gamma^\ast(\tau),\tau)\mbox{ for $0\le \tau<t$}, \\
&&L^c_{\xi}(\gamma^\ast(\tau),\tau,\gamma^\ast{}'(\tau))\in\p^+_xv(\gamma^\ast(\tau),\tau)\mbox{ for $0< \tau\le t$,}
\end{eqnarray*}
 where $\p^-_xv$ ($\p^+_xv$) is the subdifferential (superdifferential). In particular $v_x(\gamma^\ast(\tau),\tau)$ exists for $0<\tau<t$ and is equal to $L^c_{\xi}(\gamma^\ast(\tau),\tau,\gamma^\ast{}'(\tau))$.
 \item $|\gamma^\ast{}'(\tau)|\le\lambda^{-1}_1$ for $0\le \tau\le t$, where $\lambda_1$ is given in (\ref{lambda_1}).
 \item $v$ is Lipschitz continuous. 
 \item If $(x,t)$ is a regular point, then we have for any $0\le\tau< t$
 $$ u(x,t):=v_x(x,t)=\int^t_0L^c_x(\gamma^\ast(s),s,\gamma^\ast{}'(s))ds +L^c_{\xi}(\gamma^\ast(\tau),\tau,\gamma^\ast{}'(\tau)).$$ 
 If $v^0$ is semiconcave or $u^0$ is rarefaction-free, then $L^c_{\xi}(\gamma^\ast(0),0,\gamma^\ast{}'(0))=u^0(\gamma^\ast(0))$. In particular, $v$ is the viscosity solution of (\ref{HJ}) and $u$ the entropy solution of (\ref{CL}).
 \end{enumerate}
 \end{Lemma}
\begin{proof}[{\bf Proof.}]
1., 3., 4. are known, but for reader's convenience we give a brief proof.  

1. It follows from the minimizing property of $\gamma^\ast$ that we have for any $0\le\tau<t$ 
$$v(x,t)=\int^t_\tau L^c(\gamma^\ast(s),s,\gamma^\ast{}'(s))ds+v(\gamma^\ast(\tau),\tau)+h(c )(t-\tau)$$
and $\gamma^\ast|_{[0,\tau]}$ is a minimizer for $v(\gamma^\ast(\tau),\tau)$. 
Hence for any $\varepsilon>0$ such that $\tau+\varepsilon\le t$ we have 
$$v(\gamma^\ast(\tau+\varepsilon),\tau+\varepsilon)=\int^{\tau+\varepsilon}_{\tau}L^c(\gamma^\ast(s),s,\gamma^\ast{}'(s))ds+v(\gamma^\ast(\tau),\tau)+h(c )\varepsilon.$$
Define $\gamma:[0,\tau+\varepsilon]\to\R$ as $\gamma(s):=\gamma^\ast(s)+\frac{\tau+\varepsilon-s}{\varepsilon}\delta$ with $\delta\neq0$ for $s\in[\tau,\tau+\varepsilon]$ and
 $\gamma(s):=\tilde{\gamma}^\ast(s)$ for $s\in[0,\tau]$, where $\tilde{\gamma}^\ast$ is a minimizing curve for $v(\gamma^\ast(\tau)+\delta,\tau)$. Then we have 
 $$v(\gamma^\ast(\tau+\varepsilon),\tau+\varepsilon)\le\int^{\tau+\varepsilon}_{\tau}L^c(\gamma(s),s,\gamma'(s))ds+v(\gamma^\ast(\tau)+\delta,\tau)+h(c )\varepsilon.$$
 Therefore we obtain
  \begin{eqnarray*}
v(\gamma^\ast(\tau)+\delta,\tau)-v(\gamma^\ast(\tau),\tau)\ge L^c_\xi(\gamma^\ast(\tau),\tau,\gamma^\ast{}'(\tau))\delta+O(\varepsilon\delta)+O(\frac{\delta^2}{\varepsilon}),
\end{eqnarray*}
which means that for $\varepsilon=\sqrt{\delta}$
\begin{eqnarray*}
\liminf_{\delta\to0}\frac{v(\gamma^\ast(\tau)+\delta,\tau)-v(\gamma^\ast(\tau),\tau)-L^c_\xi(\gamma^\ast(\tau),\tau,\gamma^\ast{}'(\tau))\delta}{|\delta|}\ge 0.
\end{eqnarray*}
Similar reasoning with $\gamma:[0,\tau]\to\R$ defined as $\gamma(s):=\gamma^\ast(s)+\frac{s-(\tau-\varepsilon)}{\varepsilon}\delta$ for $s\in[\tau-\varepsilon,\tau]$ and $\gamma(s):=\gamma^\ast(s)$ for $s\in[0,\tau-\varepsilon]$ leads to the second inclusion. 

\indent 2. Since $\gamma^\ast$ is a minimizer, we have for $\gamma(s)\equiv x$
$$\int^t_0 L^c(\gamma^\ast(s),s,\gamma^\ast{}'(s))ds+v_0(\gamma^\ast(0))\le \int^t_0 L^c(x,s,0)ds+v_0(x)\le\alpha_2.$$
Since $\gamma^\ast$ satisfies the Euler-Lagrange equation, we obtain
\begin{eqnarray*}
|L^c_\xi(\gamma^\ast(s),s,\gamma^\ast{}'(s))|&=&|L^c_\xi(\gamma^\ast(0),0,\gamma^\ast{}'(0))+\int^t_0L_x^c(\gamma^\ast(s),s,\gamma^\ast{}'(s))ds|\\
&\le&r+\int^t_0\alpha_1(1+|L^c(\gamma^\ast(s),s,\gamma^\ast{}'(s))|)ds\\
&\le&\alpha_3+r+1.
\end{eqnarray*}
\indent 3. follows from 2. 

4. We observe that for any $0<\tau<t$ and $\delta\neq0$
\begin{eqnarray*}
v(x+\delta,t)-v(x,t)&\le&\int^t_\tau L^c(\gamma^\ast(s)+\delta,s,\gamma^\ast{}'(s))-L^c(\gamma^\ast(s),s,\gamma^\ast{}'(s))ds\\
&& +v(\gamma^\ast(\tau)+\delta,\tau)-v(\gamma^\ast(\tau),\tau).
\end{eqnarray*}
Therefore, using 1., we obtain
\begin{eqnarray*} 
u(x,t)=v_x(x,t)&\le& \int^t_\tau L^c_x(\gamma^\ast(s),s,\gamma^\ast{}'(s))ds+L^c_\xi(\gamma^\ast(\tau),\tau,\gamma^\ast{}'(\tau)),
\end{eqnarray*}
which holds for $\tau=0$.  

For the converse inequality, we take a sequence $\delta_j\to0$ such that each $(x+\delta_j,t)$ is a regular point of $v$. Let $\gamma_j$ be the minimizer of $v(x+\delta_j,t)$. Then we have $\lim_{j\to\infty}v_x(x+\delta_j,t)=v_x(x,t)$. Therefore it holds that $(\gamma_j,\gamma_j')\to(\gamma,\gamma')$ uniformly as $j\to\infty$. A similar reasoning with $\gamma_j$ yields   
\begin{eqnarray*} 
u(x,t)=v_x(x,t)&\ge& \int^t_\tau L^c_x(\gamma^\ast(s),s,\gamma^\ast{}'(s))ds+L^c_\xi(\gamma^\ast(\tau),\tau,\gamma^\ast{}'(\tau)),
\end{eqnarray*}
which holds for $\tau=0$. 
We refer to  \cite{Cannarsa} for the fact that $v$ is the viscosity solution of (\ref{HJ}) and $u$ the entropy solution of (\ref{CL}). 
\end{proof}
We state a result of hyperbolic scaling limit of random walks. We introduce a random variable $\eta(\gamma)=\{\eta^k(\gamma)\}_{k=0,1,2,\cdots,l+1}$, $\gamma\in\Omega$ which is  induced by a random walk with a velocity field $\xi$ and is defined as 
$$\eta^{l+1}:=\gamma^{l+1},\,\,\,\,\,\, \eta^k(\gamma):=\gamma^{l+1}-\sum_{k< k'\le l+1}\xi(t_{k'},\gamma^{k'})\dt\mbox{ \,\,\, for $0\le k\le l$}.$$
\begin{Lemma}(\cite{Soga2})\label{eta}
Set $\dis \tilde{\sigma}^k:=E_{\mu(\cdot;\xi)}[|\gamma^k-\eta^k(\gamma)|^2]$ and $\dis \tilde{d}^k:=E_{\mu(\cdot;\xi)}[|\gamma^k-\eta^k(\gamma)|]$ for $0\le k\le l+1$. Then we have 
$$(\tilde{d}^k)^2\le\tilde{\sigma}^k\le \frac{t^{l+1}-t^k}{\lambda}\dx.$$
\end{Lemma}

\noindent{\bf Remark.} {\it  We always have $\tilde{\sigma}^k\to0$ as $\Delta\to0$ for any $\xi$ under hyperbolic scaling. However the variance does not always tend to $0$ (see \cite{Soga2}). This estimate is analogous to the following: Let $\eta^\nu$ be the stochastic process induced by the solution $\gamma^\nu$ of (\ref{S-ODE}) as $\eta^\nu{}'(s)=\xi^\nu(s,\gamma^\nu(s)),\eta^\nu(t)=x$. Then we have for $s\in[0,t]$
\begin{eqnarray*}\label{key-estimate}
E[ | \eta^\nu(s)-\gamma^\nu(s)| ]=\sqrt{2\nu}E[|B(t-s)|]. 
\end{eqnarray*}}

\noindent{\bf Proof of Theorem \ref{Main5}.} Since $v_\Delta$ and $v$ is Lipschitz continuous,  it is enough to show that $|v^{l+1}_n-v(x_{n},t_{l+1})|=O(\sqrt{\dx})$ for all $0\le l+1\le k(T)$ and $n$. Hereafter, $\beta_1,\beta_2,\cdots$ are constants independent of $\dx,\dt,n,l$. 
Let $\gamma^\ast$ be a minimizer for $v(x_n,t_{l+1})$. By 2. of Lemma \ref{lemma-regularity}, we have $|\gamma^\ast{}'|\le\lambda_1^{-1}<\lambda^{-1}$. We take 
$$\xi(x_m,t_k):=\gamma^\ast{}'(t_k),$$
which yields a space-homogeneous random walk. Then we have 
\begin{eqnarray*}
v(x_n,t_{l+1})=\mathcal{L}(\gamma^\ast,\gamma^\ast{}')+h(c )t_{l+1},\qquad
v^{l+1}_n\le E_{\mu(\cdot;\xi)}\Big[ \mathcal{L}_\Delta(\gamma,\xi|_{\gamma})\Big]+h(c )t_{l+1}.
\end{eqnarray*}
Since $\xi$ is space-homogeneous, $\eta(\gamma)$ and $\mathcal{L}_{\Delta}(\eta(\gamma),\xi|_\gamma)$ are independent of $\gamma$ and satisfy 
\begin{eqnarray*}
|\gamma^\ast(t_k)-\eta^k(\gamma)|&\le& \beta_1 \dx\mbox{  \,\,\,for $0\le k\le l+1$},\\
|\mathcal{L}(\gamma^\ast,\gamma^\ast{}')-\mathcal{L}_{\Delta}(\eta(\gamma),\xi|_\gamma)|&\le&\beta_2\dx+|v^0(\gamma^\ast(0))-v^0_\Delta(\eta^0(\gamma))|\le\beta_3\dx.
\end{eqnarray*}
 By Lemma  \ref{eta}, we obtain 
\begin{eqnarray*}
v^{l+1}_n-v(x_{n},t_{l+1})&\le& E_{\mu(\cdot;\xi)}\Big[ \mathcal{L}_\Delta(\gamma,\xi|_{\gamma})-\mathcal{L}_{\Delta}(\eta(\gamma),\xi|_\gamma)\Big]+\beta_3\dx\\
&=&E_{\mu(\cdot;\xi)}\Big[ \sum_{0<k\le l+1}\big\{L^c(\gamma^k,t_{k-1},\gamma^\ast{}'(t_k))
-L^c(\eta^k(\gamma),t_{k-1},\gamma^\ast{}'(t_k))\big\}\dt\\
&&\qquad+v^0_\Delta(\gamma^0)-v^0_\Delta(\eta^0(\gamma))\Big]+\beta_3\dx\\
&\le&E_{\mu(\cdot;\xi)}\Big[\sum_{0<k\le l+1}\beta_4|\gamma^k-\eta^k(\gamma)|\dt+ r|\gamma^0-\eta^0(\gamma)| \Big]+\beta_3\dx\\
&\le&\beta_5\sqrt{\dx}.
\end{eqnarray*}
Let $\xi^\ast$ be the minimizer for $v^{l+1}_{n}$, which yields a space-time inhomogeneous random walk. Let $\eta_\Delta(\gamma):[0,t_{l+1}]\to\R$ be the linear interpolation of $\eta(\gamma)$. Then we have
\begin{eqnarray*}
&&v^{l+1}_n= E_{\mu(\cdot;\xi^\ast)}\Big[ \mathcal{L}_\Delta(\gamma,\xi^\ast|_{\gamma})\Big]+h(c )t_{l+1},\\
&& v(x_n,t_{l+1})\le\mathcal{L}(\eta_\Delta(\gamma),\eta_{\Delta}(\gamma)')+h(c )t_{l+1} \mbox{ \,\,\,for each $\gamma\in\Omega$},\\
&&|\mathcal{L}(\eta_\Delta(\gamma),\eta_{\Delta}(\gamma)')-\mathcal{L}_\Delta(\eta(\gamma),\xi^\ast|_{\gamma})|\le\beta_6\dx
\mbox{ \,\,\,for each $\gamma\in\Omega$}.
\end{eqnarray*}
By Lemma  \ref{eta}, we obtain 
\begin{eqnarray*}
v^{l+1}_{n}-v(x_n,t_{l+1})&\ge&E_{\mu(\cdot;\xi^\ast)}\Big[ \mathcal{L}_\Delta(\gamma,\xi^\ast|_\gamma)-\mathcal{L}_\Delta(\eta(\gamma),\xi^\ast|_{\gamma}) \Big]-\beta_6\dx\\
&=&E_{\mu(\cdot;\xi^\ast)}\Big[ \sum_{0<k\le l+1}\big\{L^c(\gamma^k,t_{k-1},\xi^\ast{}^k_{m(\gamma^k)})-L^c(\eta^k(\gamma),t_{k-1},\xi^\ast{}^k_{m(\gamma^k)})\big\}\dt \\
&&\qquad +v^0_\Delta(\gamma^0)-v^0_\Delta(\eta^0(\gamma)) \Big]-\beta_6\dx\\
&\ge&-E_{\mu(\cdot;\xi^\ast)}\Big[ \sum_{0<k\le l+1} \beta_4|\gamma^k-\eta^k(\gamma)|\dt +r|\gamma^0-\eta^0(\gamma)| \Big]-\beta_6\dx\\
&\ge&-\beta_7\sqrt{\dx}. \qquad\qquad\qquad\qquad\qquad\qquad\qquad\qquad\qquad\qquad\qquad\qquad  \Box
\end{eqnarray*}
In order to prove the convergence of minimizing random walks to minimizing curves, we observe the following lemmas: 
\begin{Lemma}\label{Lem-minimizer}
Let $\gamma^\ast$ be the unique minimizer for $v(x,t)$. We define for $\varepsilon>0,b>0$
$$\Gamma^\varepsilon:=\{\gamma:[0,t]\to\R\,|\,\gamma\in Lip,\,\,\, |\gamma(t)-\gamma^\ast(t)|\le\varepsilon,\,\,\,|\gamma'(s)|\le b,\,\,\,\mathcal{L}(\gamma,\gamma')\le \mathcal{L}(\gamma^\ast,\gamma^\ast{}')+\varepsilon \}.$$
Then we have \,\,\,$\dis\sup_{\gamma\in\Gamma^\varepsilon}\norm \gamma-\gamma^\ast\norm_{C^0([0,t])}\to0,\quad
\sup_{\gamma\in\Gamma^\varepsilon}\norm \gamma'-\gamma^\ast{}'\norm_{L^2([0,t])}\to0
 \mbox{\quad as $\varepsilon\to0$}.$
\end{Lemma}
\begin{proof}[{\bf Proof.}]
Suppose that we have a sequence $\varepsilon_j\to0$ for which there exists $\delta>0$ such that $\sup_{\gamma\in\Gamma^{\varepsilon_j}}\norm \gamma-\gamma^\ast\norm_{C^0}\ge2\delta$ for all $j$. We can take a sequence $\gamma_j\in\Gamma^{\varepsilon_j}$ for which $\norm \gamma_j-\gamma^\ast\norm_{C^0}\ge\delta$ holds for all $j$. We have a sub-sequence of $\gamma_j$, still denoted by $\gamma_j$, such that $\gamma_j\to\tilde{\gamma}$ uniformly. $\tilde{\gamma}$ satisfies $\norm \tilde{\gamma}-\gamma^\ast\norm_{C^0}\ge\delta$, $\tilde{\gamma}(t)=\gamma^\ast(t)=x$ and $\mathcal{L}(\tilde{\gamma},\tilde{\gamma}')=\mathcal{L}(\gamma^\ast,\gamma^\ast{}')$. Therefore $\tilde{\gamma}$ is another minimizer for $v(x,t)$, which is a contradiction.   

Set $\delta(\varepsilon):=\gamma^\ast(t)-\gamma(t)$ for $\gamma\in\Gamma^\varepsilon.$ Since $\gamma^\ast$ is the minimizer, we already have $\mathcal{L}(\gamma^\ast,\gamma^\ast{}')\le \mathcal{L}(\gamma+\delta(\varepsilon),\gamma')$ 
and $|\mathcal{L}(\gamma+\delta(\varepsilon),\gamma')-\mathcal{L}(\gamma,\gamma')|\le \beta_8\varepsilon$ ($\beta_8\ge1$) for any $\gamma\in\Gamma^\varepsilon$. Hence we have $\varepsilon\ge \mathcal{L}(\gamma,\gamma')-\mathcal{L}(\gamma^\ast,\gamma^\ast{}')\ge-\beta_8\varepsilon$. Therefore we obtain 
\begin{eqnarray*}
\beta_8\varepsilon&\ge&|\mathcal{L}(\gamma,\gamma')-\mathcal{L}(\gamma^\ast,\gamma^\ast{}')|\\
&=&\Big|  \int^t_0 \big\{L^c(\gamma(s),s,\gamma'(s))-L^c(\gamma^\ast(s),s,\gamma'(s))\\
&&+L^c(\gamma^\ast(s),s,\gamma'(s))-L^c(\gamma^\ast(s),s,\gamma^\ast{}'(s))\big\}ds+v^0(\gamma(0))-v^0(\gamma^\ast(0))  \Big|\\
&\ge&\Big|  \int^t_0 \big\{L^c_\xi(\gamma^\ast(s),s,\gamma^\ast{}'(s))(\gamma'(s)-\gamma^\ast{}'(s))+\frac{1}{2}L_{\xi\xi}\cdot(\gamma'(s)-\gamma^\ast{}'(s))^2\big\}ds\Big|\\
&&-\beta_9\norm\gamma-\gamma^\ast\norm_{C^0}\\
&=&\Big| L^c_\xi(\gamma^\ast(s),s,\gamma^\ast{}'(s))(\gamma(s)-\gamma^\ast{}(s))|^{s=t}_{s=0}-\int^t_0L^c_\xi(\gamma^\ast(s),s,\gamma^\ast{}'(s))'(\gamma(s)-\gamma^\ast{}(s))ds\\
&&+\int^t_0\frac{1}{2}L_{\xi\xi}\cdot(\gamma'(s)-\gamma^\ast{}'(s))^2ds\Big|-\beta_9\norm\gamma-\gamma^\ast\norm_{C^0},
\end{eqnarray*}
where we remark that $L^c_\xi(\gamma^\ast(s),s,\gamma^\ast{}'(s))'=L^c_x(\gamma^\ast(s),s,\gamma^\ast{}'(s))$ due to the Euler-Lagrange equation.  Since $L^c_{\xi\xi}$ is bounded away from zero on each compact set,   we see that for constants $\beta_{10}>0, \beta_{11}>0$
 \begin{eqnarray*}
\beta_8\varepsilon&\ge&\beta_{10}\norm \gamma'-\gamma^\ast{}'\norm_{L^2}^2
-\beta_{11}\norm\gamma-\gamma^\ast\norm_{C^0}.
 \end{eqnarray*} 
We have already proved that  $\norm\gamma-\gamma^\ast\norm_{C^0}\to0$ as $\varepsilon\to0$. This finishes the proof. 
\end{proof}
\begin{Lemma}\label{L^2-C^0}
Let $f:[0,t]\to\R$ be a Lipschitz function with a Lipschitz constant $\theta$ satisfying $f(t)=0$. Then we have
$\norm f\norm_{C^0([0,t])}\le \theta\norm f\norm_{L^2([0,t])}+\sqrt{\norm f\norm_{L^2([0,t])}}.$
\end{Lemma}
\begin{proof}[{\bf Proof.}]
We suppose that $\norm f\norm_{C^0}>\sqrt{\norm f\norm_{L^2}}$. 
Otherwise we get the conclusion. 
For $I:=\{s\in[0,t]\,\,|\,\, |f(s)|\ge \sqrt{\norm f\norm_{L^2}}\}$, we have  
$\norm f\norm_{L^2}^2\ge \int _I|f(s)|^2ds\ge \norm f\norm_{L^2}|I|$ and 
$|I|\le \norm f\norm_{L^2}$.  The set $I$ necessarily contains $s_\ast$ such that $|f(s_\ast)|=\norm f\norm_{C^0}$. Since $f(t)=0$, there always exists $s_0>s_\ast$ such that $f(s_0)=\sqrt{\norm f\norm_{L^2}}$ and $[s_\ast,s_0]\subset I$. Hence we have 
$|f(s_0)-f(s_\ast)|\le \theta|s_0-s_\ast|\le \theta |I|$. Therefore we obtain 
$| f(s_\ast)|\le \theta |I|+|f(s_0)|.$ 
\end{proof}
\begin{proof}[{\bf Proof of Theorem \ref{Main6}.}]
Consider  $\Omega^\varepsilon_\Delta:=\{\gamma\in\Omega\,|\, 
\norm \gamma_\Delta-\gamma^\ast \norm_{C^0}\le\varepsilon \}$, 
where $\gamma_\Delta$ is the linear interpolation of $\gamma$ extended to $s\in[0,t]$.  We show that, for each arbitrarily fixed $\varepsilon>0$, we have 
$prob(\Omega^\varepsilon_\Delta)\to1$ as $\Delta\to0$.    

First of all we prove the following:  Introduce  
$$\Gamma^\varepsilon_\Delta:=\{\gamma\in\Omega\,|\, 
\norm \eta_\Delta(\gamma)-\gamma^\ast \norm_{C^0}\le\varepsilon\mbox{ and }\norm \eta_\Delta(\gamma)'-\gamma^\ast{}' \norm_{L^2}\le\varepsilon \},$$
where $\eta_\Delta(\gamma)$ is the linear interpolation of $\eta(\gamma)$ extended to $s\in[0,t]$.  Then, for each arbitrarily fixed $\varepsilon>0$, we have 
$prob(\Gamma^\varepsilon_\Delta)\to1$  as $\Delta\to0$. In fact, we observe that
 \begin{eqnarray*}
 v^{l+1}_n
 &=&E_{\mu(\cdot;\xi^\ast)}\Big[\mathcal{L}_\Delta(\eta(\gamma),\xi^\ast|_{\gamma})\Big]+h(c )t_{l+1}+O(\sqrt{\dx})\\
 &=&E_{\mu(\cdot;\xi^\ast)}\Big[\mathcal{L}(\eta_\Delta(\gamma),\eta_\Delta(\gamma)')\Big]+h(c )t_{l+1}+O(\sqrt{\dx}),\\
v^{l+1}_n-v(x,t)&=& E_{\mu(\cdot;\xi^\ast)}\Big[\mathcal{L}_\Delta(\gamma,\xi^\ast|_\gamma)-\mathcal{L}(\gamma^\ast,\gamma^\ast{}')\Big]+h(c )(t_{l+1}-t)\\
&=& E_{\mu(\cdot;\xi^\ast)}\Big[\mathcal{L}(\eta_\Delta(\gamma),\eta_\Delta(\gamma)')-\mathcal{L}(\gamma^\ast,\gamma^\ast{}')\Big]+O(\sqrt{\dx})\\
&=&O(\sqrt{\dx}).
 \end{eqnarray*}
Since $\gamma^\ast$ is the minimizer, we have  for  all $\gamma\in\Omega$ 
$$ 0\le \mathcal{L}(\eta_\Delta(\gamma)+x-x_n,\eta_\Delta(\gamma)')-\mathcal{L}(\gamma^\ast,\gamma^\ast{}')=\mathcal{L}(\eta_\Delta(\gamma),\eta_\Delta(\gamma)')-\mathcal{L}(\gamma^\ast,\gamma^\ast{}')
+\beta_{12}\dx.$$
Consider $\Omega^+:=\{\gamma\in\Omega\,|\, \mathcal{L}(\eta_\Delta(\gamma),\eta_\Delta(\gamma)')-\mathcal{L}(\gamma^\ast,\gamma^\ast{}')\ge \dx^{\frac{1}{4}}\}$. Then we obtain
\begin{eqnarray*}
O(\sqrt{\dx})=E_{\mu(\cdot;\xi^\ast)}\Big[\mathcal{L}(\eta_\Delta(\gamma),\eta_\Delta(\gamma)')-\mathcal{L}(\gamma^\ast,\gamma^\ast{}')\Big]
\ge\sum_{\gamma\in\Omega^+}\mu(\gamma;\xi^\ast)\dx^{\frac{1}{4}}-\beta_{12}\dx.
\end{eqnarray*}
Therefore $prob(\Omega^+)=O(\dx^{\frac{1}{4}})$. It follows from Lemma \ref{Lem-minimizer} that  for $\dx\ll\varepsilon$ 
$$\Omega\setminus\Omega^+\subset \Gamma^\varepsilon_\Delta.$$
Thus we conclude that $1-prob(\Omega^+)\le prob(\Gamma^\varepsilon_\Delta)\le1$ and $prob(\Gamma^\varepsilon_\Delta)\to1$ as $\Delta\to0$.

Since $\eta(\gamma)$ converges to $\gamma^\ast$ uniformly in probability, for any $\varepsilon'>0$ there exists $\delta>0$ such that if $|\Delta|<\delta$ we have
 \begin{eqnarray*}
 E_{\mu(\cdot;\xi^\ast)}[|\gamma^k- \gamma^\ast(t_k)|]&\le&E_{\mu(\cdot;\xi^\ast)}[|\gamma^k- \eta^k(\gamma)|]
 +  E_{\mu(\cdot;\xi^\ast)}[|\eta^k(\gamma)- \gamma^\ast(t_k)|] \le\varepsilon'\mbox{  for all $k$},\\
 E_{\mu(\cdot;\xi^\ast)}[\norm \gamma_\Delta-\gamma^\ast \norm_{L^2}]&\le& \varepsilon'.
 \end{eqnarray*}
 Define $\Omega^{++}:=\{\gamma\in\Omega\,|\, \norm \gamma_\Delta-\gamma^\ast \norm_{L^2}\ge\sqrt{\varepsilon'}\}$. Since 
 $\varepsilon'\ge E_{\mu(\cdot;\xi^\ast)}[\norm \gamma_\Delta-\gamma^\ast \norm_{L^2}]\ge prob(\Omega^{++})\sqrt{\varepsilon'} ,$ 
we have $prob(\Omega^{++})\le \sqrt{\varepsilon'}$. 
  Take $\varepsilon'=O(\varepsilon^4)$. By Lemma \ref{L^2-C^0} we have $\norm\gamma_\Delta-\gamma^\ast\norm_{C^0}\le \varepsilon$ for all $\gamma\in \Omega\setminus \Omega^{++}$.  Therefore we conclude that $\Omega\setminus\Omega^{++}\subset\Omega^\varepsilon_\Delta$, $1-\sqrt{\varepsilon'}\le prob(\Omega^\varepsilon_\Delta)\le1$ and $prob(\Omega^\varepsilon_\Delta)\to1$ as $\Delta\to0$.
\end{proof}
\begin{proof}[{\bf Proof of Theorem \ref{Main7}.}]
Let $(x_n,t_{l+1})$ be such that $x\in[x_n-\dx,x_n+\dx),\,\,t\in[t_{l+1},t_{l+2})$. It is enough to show that
$u^{l+1}_{n+1}\to u(x,t)\mbox{ as $\Delta\to0$}.$ We set 
$$\mathcal{L}'(\gamma;\gamma'):=\int^t_0L_x(\gamma(s),s,\gamma'(s))ds.$$
\indent First of all we assume that $v^0$ is semiconcave, or $u^0$ is rarefaction-free.    Let $\xi^\ast$ be the minimizer for $v^{l+1}_{n}$. It follows from Theorem \ref{Main4} that 
$$u^{l+1}_{n+1}\le E_{\mu(\cdot;\xi)}\Big[\mathcal{L}'(\eta_\Delta(\gamma);\eta_\Delta(\gamma)')+u^0_\Delta(\gamma^0+\dx)\Big]+O(\sqrt{\dx}),$$
where $\eta_\Delta(\gamma)$ is the linear interpolation of $\eta(\gamma)$ extended to $s\in[0,t]$.
Therefore, by 4. of Lemma \ref{lemma-regularity}, we obtain
\begin{eqnarray*}
u^{l+1}_{n+1}-u(x,t)&\le& E_{\mu(\cdot;\xi^\ast)}\Big[\mathcal{L}'(\eta_\Delta(\gamma);\eta_\Delta(\gamma)')-\mathcal{L}'(\gamma^\ast;\gamma^\ast{}')\\
&&+u^0_\Delta(\gamma^0+\dx)-u^0(\gamma^\ast(0)) \Big]+O(\sqrt{\dx}),\\
u^0_\Delta(\gamma^0+\dx)-u^0(\gamma^\ast(0))&=&\frac{1}{2\dx}\int_{(\gamma^0+\dx)-\dx}^{(\gamma^0+\dx)+\dx}u^0(y)-u^0(\gamma^\ast(0))dy.
\end{eqnarray*}
Since  $\norm\eta_\Delta(\gamma)-\gamma^\ast\norm_{C^0}\to0$,  $\norm\eta_\Delta(\gamma)'-\gamma^\ast{}'\norm_{L^2}\to0$, $\norm\gamma_\Delta-\gamma^\ast\norm_{C^0}\to0$ as $\Delta\to0$ in probability and $u^0(y)\to u^0(\gamma^\ast(0))$ as $y\to\gamma^\ast(0)$, we obtain
\begin{eqnarray}\label{sup}
\limsup_{\Delta\to0}(u^{l+1}_{n+1}-u(x,t))=0.
\end{eqnarray}
Similar reasoning with the converse inequality in Theorem \ref{Main4} yields  
\begin{eqnarray}\label{inf}
\liminf_{\Delta\to0}(u^{l+1}_{n+1}-u(x,t))=0.
\end{eqnarray}
\indent We remove the assumption that $v^0$ is semiconcave, or $u^0$ is rarefaction-free. Hence we do not always have $u^0(\gamma^\ast(0))$ (in this case, $u^0$ jumps up at $x=\gamma^\ast(0)$).  By 4. of Lemma \ref{lemma-regularity}  and Theorem \ref{Main4}, we see that for any $0\le\tau<\min\{t_{l+1},t\}$ 
\begin{eqnarray*}
u(x,t)&=&\int^t_\tau L^c_x(\gamma^\ast(s),s,\gamma^\ast{}'(s))ds+L^c_{\xi}(\gamma^\ast(\tau),\tau,\gamma^\ast{}'(\tau)),\\
u^{l+1}_{n+1}&\le&E_{\mu(\cdot;\xi^\ast)}\Big[\sum_{k(\tau)<k\le l+1}L^c_x(\gamma^{k},t_{k-1},\xi^\ast{}^k_{m(\gamma^k)})\dt+D_xv^{k(\tau)}_{m(\gamma^{k(\tau)})+1} \Big]+O(\dx)\\
&=&E_{\mu(\cdot;\xi^\ast)}\Big[\sum_{k(\tau)<k\le l+1}L^c_x(\gamma^{k},t_{k-1},\xi^\ast{}^k_{m(\gamma^k)})\dt+L^c_\xi(\gamma^{k(\tau)},t_{k(\tau)},\xi^\ast{}^{k(\tau)+1}_{m(\gamma^{k(\tau)})} ) \Big]+O(\dx)\\
&=&E_{\mu(\cdot;\xi^\ast)}\Big[ \int^t_\tau L^c_x(\eta_\Delta(\gamma)(s),s,\eta_\Delta(\gamma)'(s))ds\\
&&\qquad\qquad\qquad\qquad\qquad\qquad\qquad +L^c_{\xi}(\eta_\Delta(\gamma)(\tau),\tau,\eta_\Delta(\gamma)'(\tau))\Big]+O(\sqrt{\dx}).
\end{eqnarray*} 
The average of $\eta_\Delta(\gamma)'$, denoted by $\bar{\eta}_\Delta'$, converges to $\gamma^\ast{}'$ in $L^2$. Therefore for any $\varepsilon>0$ there exists $\delta>0$ such that if $|\Delta|<\delta$ we have $\norm\bar{\eta}_\Delta'-\gamma^\ast{}'\norm_{L^2}<\varepsilon$. In particular we have $0\le\tau^\ast<\min\{t_{l+1},t\}$ such that $|\bar{\eta}_\Delta'(\tau^\ast)-\gamma^\ast{}'(\tau^\ast)|<\varepsilon$. Taking $\tau=\tau^\ast$, we conclude (\ref{sup}). Similar reasoning yields (\ref{inf}).
\end{proof}

\noindent{\bf Acknowledgements.} I would like to thank Prof. Takaaki Nishida and Prof. Wilhelm Stannat, who have interests in this work and listened to me. I am profited by ``International Research Training Group 1529  Mathematical Fluid Dynamics'' associated with Technische Universit\"at Darmstadt and Waseda University.


\begin{thebibliography}{9}
%
\bibitem{Bessi}
U. Bessi, Aubry-Mather theory and Hamilton-Jacobi equations, Commun. Math. Phys.  \bf 235 \rm(2003), 495-511.
%
\bibitem{Cannarsa}
P. Cannarsa and C. Sinestrari, Semiconcave functions, Hamilton-Jacobi equations and optimal control, Birkh\"auser \rm(2004).
%
\bibitem{Crandall-Lions-approximation}
M. G. Crandall and P. L. Lions, Two approximations of solutions of Hamilton-Jacobi equations,  Math. Comp.  \bf43  \rm(1984),  No. 167, 1-19.
%
\bibitem{Crandall-Majda}
M. G. Crandall and Majda, Monotone difference approximations for scalar conservation laws,  Math. Comp.  \bf34  \rm(1980), No. 149, 1-21.
%
\bibitem{WE}
W. E, Aubry-Mather theory and periodic solutions of the forced Burgers equation, \rm Comm. Pure Appl. Math. \bf 52 \rm (1999), No. 7, 811-828.
%
\bibitem{Fathi97-1}
A. Fathi, Th\'eor\`eme KAM faible et th\'eorie de Mather sur les syst\`emes lagrangiens, (French) [A weak KAM theorem and Mather's theory of Lagrangian systems] \rm C. R. Acad. Sci. Paris S\'er. I 
Math. \bf 324 \rm (1997), No. 9, 1043-1046.
\bibitem{Fathi-book}
A. Fathi, Weak KAM theorem in Lagrangian dynamics,  Cambridge Univ. Pr. (2011).
%
\bibitem{Fleming}
W. H. Fleming, The Cauchy problem for a nonlinear first order partial differential equation,  \rm J. Differ. Eqs  \bf5  \rm(1969), 515-530.
%
\bibitem{Fleming-Soner}
W. H. Fleming and H. M. Soner,  Controlled Markov Processes and Viscosity Solutions, Springer (2006).
\bibitem{JKM}
H. R. Jauslin, H. O. Kreiss and J. Moser, On the forced Burgers equation with periodic boundary conditions, \rm Proc. Sympos. Pure Math. \bf 65 \rm (1999), 133-153.
\bibitem{Nishida-Soga}
T. Nishida and K. Soga, Diffenrence approximation to Aubry-Mather sets of the forced Burgers equation, submitted. 
%
\bibitem{Oleinik}
O. A. Oleinik, Discontinuous solutions of nonlinear differential equations, \rm A. M. S. Transl. (ser. 2) \bf 26 \rm (1957), 95-172.
%
\bibitem{Sabac}
F. \c{S}abac, The optimal convergence rate of monotone finite difference methods for hyperbolic conservation laws, \rm SIAM. J. Numer. Anal.
\bf 34 \rm (1997), No. 6, 2306-2318.
%
\bibitem{Sob}
A. N. Sobolevskii, Periodic solutions of the Hamilton-Jacobi equation with a periodic non-homogeneous term and Aubry-Mather theory, Sb. Mat. \bf 190 \rm(1999), 1487-1504.
%
\bibitem{Soga2}
K. Soga, Space-time continuous limit of random walk with hyperbolic scaling, preprint. 
%
 \bibitem{Souganidis}
 P. E. Souganidis, Approximation schemes for viscosity solutions of Hamilton-Jacobi equations,  \rm J. Differ. Eqs. \bf 59 \rm (1985), 1-43.
 %
\bibitem{Tadmor1}
E. Tadmor, The large-time behavior of the scalar, genuinely nonlinear Lax-Friedrichs scheme,  Math. Comp.  \bf 43  \rm(1984),  No. 168, 353-368.
%
\end{thebibliography}
\end{document}